\documentclass[a4paper,12pt,reqno]{amsart}
\usepackage{amsmath,amsthm,amssymb}
\usepackage[numbers,sort&compress]{natbib}
\nonstopmode \numberwithin{equation}{section}

\newtheorem{theorem}{Theorem}[section]
\newtheorem{proposition}{Proposition}[section]

\newtheorem{lemma}{Lemma}[section]

\allowdisplaybreaks
\DeclareMathOperator{\real}{Re}

\allowdisplaybreaks

\begin{document}
\title[Inequalities for $k-$Bessel function]{Inequalities for the modified $k$-
Bessel function}

\author{Saiful. R. Mondal}
\address{Department of Mathematics\\
King Faisal University, Al Ahsa 31982, Saudi Arabia}
\email{saiful786@gmail.com}

\author{Kottakkaran S. Nisar}
\address{Department of Mathematics\\
Prince Sattam bin Abdulaziz University, Saudi Arabia}
\email{ksnisar1@gmail.com}

\begin{abstract}
The article considers the  generalized $k$-Bessel functions and represents it as Wright functions. Then we study the monotonicity properties of the ratio of two different orders $k$- Bessel functions, and the ratio of the $k$-Bessel and the $m$-Bessel functions. The log-convexity with respect to the order of the $k$-Bessel also given. An investigation regarding the monotonicity of the ratio of the $k$-Bessel and $k$-confluent hypergeometric functions are discussed.
\end{abstract}
\maketitle
\section{Introduction}
\label{Intro}
 One of the generalization of the classical gamma function $\Gamma$
studied in \cite{Diaz} is defined by the limit formula
\begin{align}\label{eqn-1}
\Gamma_k(x) := \lim_{n \to \infty} \frac{n! \; k^n (n^k)^{\tfrac{x}{k}-1}}{(x)_{n, k}}, \quad k>0,
\end{align}
where $(x)_{n, k}:=x(x+k) (x+2k)\ldots (x+(n-1)k)$ is called $k$-Pochhammer symbol. The above $k-$gamma function also
have an integral representation as
\begin{align}\label{eqn-2}
\Gamma_k(x)= \int_0^\infty t^{x-1} e^{-\frac{t^{k}}{k}} dt, \quad \real(x)>0.
\end{align}
Properties of the $k$-gamma functions have been studies by many researchers \cite{CGK,CGK2,VK, MM, Mubeen13}. Follwoing properties are required in sequel:
\begin{itemize}
\item[(i)]
$\Gamma _{k}\left( x+k\right) =x\Gamma _{k}\left( x\right)$
\item[(ii)]
$\Gamma _{k}\left( x\right) =k^{\frac{x}{k}-1}\Gamma \left( \frac{x}{k}\right)$
\item[(iii)] $\Gamma _{k}\left( k\right)=1$
\item[(iv)]$\Gamma _{k}\left( x+nk\right)=\Gamma _{k}(x) (x)_{n, k}$
\end{itemize}

Motivated with the above generalization of the $k$-gamma functions, Romero et. al.\cite{Romero-Cerutti} introduced the
$k-$Bessel function of the first kind defined by the series
\begin{align}\label{k1}
J_{k,\nu }^{\gamma ,\lambda }\left( x\right) :=\sum_{n=0}^{\infty }\frac{%
\left( \gamma \right) _{n,\;k}}{\Gamma _{k}\left( \lambda n+\upsilon +1\right)
}\frac{\left( -1\right) ^{n}\left( x/2\right) ^{n}}{\left( n!\right) ^{2}},
\end{align}%
where $k\in \mathbb{R^{+}}$; $\alpha,\lambda,\gamma,\upsilon \in C$; $\real(\lambda)>0$ and $\real(\upsilon) >0$. They also established two recurrence relations for $J_{k,\nu }^{\gamma ,\lambda }$.

In this article, we are considering the following function:
\begin{align}\label{eqn-modfb}
I_{k,\nu }^{\gamma ,\lambda }\left( x\right) :=\sum_{n=0}^{\infty }\frac{%
\left( \gamma \right) _{n,\;k}}{\Gamma _{k}\left( \lambda n+\upsilon +1\right)
}\frac{\left( x/2\right) ^{n}}{\left( n!\right) ^{2}},
\end{align}
Since
\[\lim_{k, \lambda, \gamma \to 1 }I_{k,\nu }^{\gamma ,\lambda }\left( x\right)=
\sum_{n=0}^{\infty }\frac{%
1}{\Gamma\left( n+\upsilon +1\right)
}\frac{\left( x/2\right) ^{n}}{n!}=\left(\frac{2}{x}\right)^{\frac{\nu}{2}} I_\nu(\sqrt{2x}),
\]
the classical modified Bessel functions of first kind. In this sense, we can call $I_{k,\nu }^{\gamma ,\lambda }$ as the modified $k$-Bessel functions of first kind. In fact, we can express both $J_{k,\nu }^{\gamma ,\lambda }$ and $I_{k,\nu }^{\gamma ,\lambda }$ together in
\begin{equation}\label{eqn-genb}
\mathtt{W}_{k,\nu, c }^{\gamma,\lambda }(x) :=\sum_{n=0}^{\infty}\frac{
( \gamma )_{n,\;k}}{\Gamma_{k}(\lambda n+ \nu +1)
}\frac{(-c)^n( x/2) ^{n}}{\left( n!\right) ^{2}}, \quad c \in \mathbb{R}.
\end{equation}
We can termed $\mathtt{W}_{k,\nu }^{\gamma ,\lambda }$  as the generalized $k$-Bessel function.

First we study the representation formulas for $\mathtt{W}_{k,\nu }^{\gamma ,\lambda }$
in term of the classical Wright functions.
Then we will study about
the monotonicity and log-convexity properties of $I_{k,\nu }^{\gamma ,\lambda }$.

\section{Representation formula for the generalized $k$-Bessel function}
\label{sec1}

The generalized hypergeometric function ${}_pF_q(a_1,\ldots, a_p;c_1,\ldots,c_q;x)$, is given by the power series
\begin{equation}\label{eqn:gen-hyp-func}
{}_pF_q(a_1,\ldots, a_p;c_1,\ldots,c_q;z) = \sum_{k=0}^{\infty}\dfrac{(a_1)_{k}
\cdots (a_p)_{k}}{(c_1)_{k}\cdots(c_q)_{k}(1)_{k}}z^k, \quad \quad |z|<1,
\end{equation}
where the $c_{i}$ can not be zero or a negative integer. Here $p$ or $q$ or both are allowed to be zero.
The series $(\ref{eqn:gen-hyp-func})$ is absolutely convergent for all finite $z$  if $p\leq q$ and for $|z|<1$ if $p=q+1$.
When $p>q+1$, then the series diverge for $z \not=0$ and the series does not terminate.

 The generalized Wright hypergeometric function ${}_p\psi_q(z)$ is given by the series
\begin{equation}\label{eqn-9-bessel}
{}_p\psi_q(z)={}_p\psi_q\left[\begin{array}{c}
(a_i,\alpha_i)_{1,p} \\
(b_j,\beta_j)_{1,q}
\end{array}\bigg|z\right]=\displaystyle\sum_{k=0}^{\infty}\frac{\prod_{i=1}^{p}\Gamma(a_{i}+\alpha_{i}k)}
{\prod_{j=1}^{q}\Gamma(b_{j}+\beta_{j}k)}
\frac{z^{k}}{k!},
\end{equation}
where $a_i, b_j\in \mathbb{C}$, and real $\alpha_i, \beta_j\in \mathbb{R}$ ($i=1,2,\ldots,p; j =1,2,\ldots,q$).
The asymptotic behavior of this function for large values of argument of $z\in \mathbb{C}$
were studied in \cite{CFox, Kilbas} and under the condition
\begin{equation}\label{eqn-10-bessel}
\displaystyle\sum_{j=1}^{q}\beta_{j}-\displaystyle\sum_{i=1}^{p}\alpha_{i}>-1
\end{equation}
in literature \cite{Wright-2, Wright-3}. The more properties of the Wright function are investigated in
\cite{Kilbas, Kilbas-itsf, KST}.

Now we will give the representation of the generalized $k$-Bessel functions in terms of the Wright and generalized
hypergeometric functions.
\begin{proposition}
Let,  $k \in \mathbb{R}$ and $\lambda ,\gamma ,\nu \in \mathbb{C}$ such that $\real(
\lambda) >0,\real( \nu ) >0.$ Then
\begin{equation*}
\mathtt{W}_{k,\nu,c }^{\gamma ,\lambda }(x) =\frac{1}{k^{\frac{\nu+k+1}{k}}\Gamma\left(\frac{\gamma}{k}\right)}
{}_1\psi_2\left[\begin{array}{ccc}
\left(\frac{\gamma}{k}, 1\right)& \\
\left(\frac{\nu+1}{k}, \frac{\gamma}{k}\right) & (1, 1)
\end{array}\bigg|-\frac{c x}{2 k^{\frac{\lambda}{k}-1}}\right]
\end{equation*}
\end{proposition}
\begin{proof}
Using the relations $\Gamma _{k}\left( x\right) =k^{\frac{x}{k}-1}\Gamma \left( \frac{x}{k}\right)$
and
$\Gamma _{k}\left( x+nk\right)=\Gamma _{k}(x) (x)_{n, k}$,
the generalized $k$-Bessel functions defined in  $\eqref{eqn-genb}$  can be rewrite as
\begin{align}
\mathtt{W}_{k,\nu,c }^{\gamma ,\lambda }(x)
&=\sum_{n=0}^\infty \frac{\Gamma_k(\gamma+n k)}{\Gamma_k(\lambda n+\nu+1) \Gamma_k(\gamma)} \frac{(-c)^n}{(n!)^2} \left(\frac{x}{2}\right)^n\\
&=\frac{1}{k^{\frac{\nu+k+1}{k}} \Gamma\left(\frac{\gamma}{k}\right)}\sum_{n=0}^\infty \frac{\Gamma\left(\frac{\gamma}{k}+n \right)}{\Gamma\left(\frac{\lambda}{k} n+\frac{\nu+1}{k}\right) \Gamma\left(\frac{\gamma}{k}\right)} \frac{(-c)^n}{\Gamma(n+1)\Gamma(n+1)} \left(\frac{x}{2 k^{\frac{\lambda}{k}-1}}\right)^n\\
&=\frac{1}{k^{\frac{\nu+k+1}{k}}\Gamma\left(\frac{\gamma}{k}\right)}
{}_1\psi_2\left[\begin{array}{ccc}
\left(\frac{\gamma}{k}, 1\right)& \\
\left(\frac{\nu+1}{k}, \frac{\gamma}{k}\right) & (1, 1)
\end{array}\bigg|-\frac{c x}{2 k^{\frac{\lambda}{k}-1}}\right]
\end{align}
Hence the result follows.
\end{proof}

\section{Monotonicty and log-convexity properties}
\label{sec2}

This section discuss the monotonicity and log-convexity properties for the modified $k$-Bessel functions $\mathtt{W}_{k,\nu,-1 }^{\gamma ,\lambda }(x)=\mathtt{I}_{k,\nu}^{\gamma ,\lambda }(x)$.

Following lemma due to Biernacki and Krzy\.z \cite{Biernacki-Krzy} will be required.

\begin{lemma}\label{lemma:1}\cite{Biernacki-Krzy}
Consider the power series $f(x)=\sum_{k=0}^\infty a_k x^k$ and $g(x)=\sum_{k=0}^\infty b_k x^k$, where $a_k \in \mathbb{R}$ and $b_k > 0$ for all $k$. Further suppose that both series converge on $|x|<r$. If the sequence $\{a_k/b_k\}_{k\geq 0}$ is increasing (or decreasing), then the function $x \mapsto f(x)/g(x)$ is also increasing (or decreasing) on $(0,r)$.
\end{lemma}
The above lemma still holds when both $f$ and  $g$ are even, or both are odd functions.

\begin{theorem}
The following results  holds true for the modified $k$-Bessel functions.
\begin{enumerate}
\item For $\mu \geq \nu>-1$, the function $x \mapsto \mathtt{I}_{k,\mu}^{\gamma ,\lambda }(x)/\mathtt{I}_{k,\nu}^{\gamma ,\lambda }(x)$ is increasing on $(0, \infty)$ for some fixed $k >0$.
\item If $k\geq \lambda \geq m>0$, the function $x \mapsto \mathtt{I}_{k,\nu}^{\gamma ,\lambda }(x)/\mathtt{I}_{m,\nu}^{\gamma ,\lambda }(x)$ is increasing on $(0, \infty)$ for some fixed $\nu >-1$ and $\gamma \geq \nu+1$. 
\item The function $\nu \mapsto \mathcal{I}_{k,\nu}^{\gamma ,\lambda }(x)$ is log-convex on $(0, \infty)$ for some fixed $k, \gamma>0$ and $x>0$. Here, $\mathcal{I}_{k,\nu}^{\gamma ,\lambda }(x):=\Gamma_k(\nu+1)\mathtt{I}_{k,\nu}^{\gamma ,\lambda }(x)$.
\item Suppose that $\lambda \geq k>0$ and $\nu>-1$. Then
 \begin{enumerate}
 \item The function $x \mapsto \mathtt{I}_{k,\nu}^{\gamma ,\lambda }(x)/\Phi _{k}\left( a,c;x\right)$ is decreasing on $(0, \infty)$ for $a \geq c >0$ and  $0<\gamma \leq \nu+1$. Here, $\Phi _{k}\left( a; c; x\right)$ is the $k$-confluent hypergeometric functions.
 \item  The function $x \mapsto \mathtt{I}_{k,\nu}^{\gamma ,\lambda }(x)/\Phi _{k}\left( \gamma; \lambda; x/2\right)$ is decreasing on $(0, 1)$ for $\gamma >0$ and  $0< k \leq \lambda \leq \nu+1$.
  \item The function $x \mapsto \mathtt{I}_{k,\nu}^{\gamma ,\lambda }(x)/\Phi _{k}\left( \gamma; \lambda; x/2\right)$ is decreasing on $[1, \infty)$ for $\gamma >0$ and  $0< k \leq \min\{\lambda,\nu+1\}$.
 \end{enumerate}
\end{enumerate}
\end{theorem}
\begin{proof}
${\bf (1)}$ Form \eqref{eqn-modfb} it follows that
\[\mathtt{I}_{ k, \nu}^{\gamma ,\lambda }(x)= \sum_{n=0}^\infty a_n(\nu) x^n\quad
\text{and} \quad \mathtt{I}_{ k, \nu}^{\gamma ,\lambda }(x)= \sum_{n=0}^\infty a_n(\mu) x^n,\]
where
\[a_n(\nu)= \frac{(\gamma)_{n,k}}{\Gamma_k(\lambda n+\nu+1) (n!)^2 2^n}
\quad \text{and} \quad
a_n(\mu)= \frac{(\gamma)_{n,k}}{\Gamma_k(\lambda n+\mu+1) (n!)^2 2^n}
\]
Consider the function
\[ f(t):= \frac{\Gamma_k(\lambda t+\mu+1)}{\Gamma_k(\lambda t+\nu+1)}.\]
Then the logarithmic differentiation yields
\begin{align*}
\frac{f'(t)}{f(t)}= \lambda( \Psi_k(\lambda t+\mu+1)-\Psi_k(\lambda t+\nu+1)).
\end{align*}
Here, $\Psi_k=\Gamma_k'/\Gamma_k $ is the $k$-digamma functions studied in \cite{Kwara14} and defined by
\begin{align}\label{def-digamma}
\Psi_k(t)=\frac{\log(k)-\gamma_1}{k}-\frac{1}{t}+\sum_{n=1}^\infty \frac{t}{nk(nk+t)}\end{align}
where $\gamma_1$ is the Euler-Mascheroni’s constant.

A calculation yields
\begin{align}\label{def-digamma-2} \Psi_k'(t)=\sum_{n=0}^\infty  \frac{1}{(nk+t)^2}, \quad k>0 \quad \text{and} \quad t>0.\end{align}
Clearly, $\Psi_k$ is increasing on $(0, \infty)$ and hence $f'(t)>0$ for all $t\geq0$ if $\mu \geq \nu>-1$. This, in particular, implies that the sequence $\{d_n\}_{n \geq 0}=\{a_n(\nu)/a_n(\mu)\}_{n \geq 0}$ is increasing and hence the conclusion follows from Lemma $\ref{lemma:1}$.

{\bf (2)}. This result also follows from Lemma $\ref{lemma:1}$ if the sequence
$\{d_n\}_{n \geq 0}=\{a_n^k(\nu)/a_n^m(\mu)\}_{n \geq 0}$ is increasing for $k \geq m >0$.
Here,
\[
a_{n}^{k}\left( \nu \right) =\frac{\left( \gamma \right) _{n,k}}{%
\Gamma _{k}\left( \lambda n+\nu+1\right) \left( n!\right) ^{2}} \quad \text{and} \quad
a_{n}^{m}\left( \nu \right)  =\frac{\left( \gamma \right) _{n,m}}{%
\Gamma _{m}\left( \lambda n+\nu+1\right) \left( n!\right) ^{2}},
\]
which together with the identity $\Gamma _{k}\left( x+nk\right)=\Gamma _{k}(x) (x)_{n, k}$ gives
\begin{align*}
d_n&=\frac{\left( \gamma \right) _{n,k}}{\left( \gamma \right) _{n,m}} \frac{
\Gamma _{m}\left( \lambda n+\nu+1\right) }{\Gamma _{k}\left( \lambda n+\nu+1\right)}\\
&= \frac{
\Gamma _{k}\left( \gamma +nk\right)\Gamma _{m}\left( \lambda n+\nu+1\right) }{\Gamma _{k}\left( \gamma +nm\right)\Gamma _{k}\left( \lambda n+\nu+1\right)}.
\end{align*}
Now to show that $\{d_n\}$ is increase, consider the function
\[ f(y):=\frac{
\Gamma _{k}\left( \gamma +yk\right)\Gamma _{m}\left( \lambda y+\nu+1\right) }{\Gamma _{k}\left( \gamma +ym\right)\Gamma _{k}\left( \lambda y+\nu+1\right)}\]
The logarithmic differentiation of $f$ yields
\begin{align}\label{3}
\frac{f'(y)}{f(y)}= k \Psi_k(\gamma +yk)+ \lambda \Psi_m\left( \lambda y+\nu+1\right)-m  \Psi_m(\gamma +ym )-\lambda \Psi_k\left( \lambda y+\nu+1\right)
\end{align}
If $\gamma \geq  \nu+1$ and $k \geq \lambda \geq m $, then \eqref{3} can be rewrite as
\begin{align}\label{44}
\frac{f'(y)}{f(y)}\geq  \lambda \big(\Psi_k(\nu+1 +yk)- \Psi_k\left( \lambda y+\nu+1\right)\big)+ m\big( \Psi_m\left( \lambda y+\nu+1\right)- \Psi_m(\nu+1 +ym )\big) \geq 0.
\end{align}
This conclude that $f$, and consequently the sequence $\{d_n\}_{n\geq 0}$, is increasing. Finally the result follows from the Lemma \ref{lemma:1}.

{\bf (3).} It is known that sum of the log-convex functions is log-convex. Thus, to prove the result it is enough  to show that
\[
\nu \mapsto a_{n}^{k}\left( \nu \right) :=\frac{\left( \gamma \right) _{n,k}\Gamma _{k}\left( \nu+1\right)}{%
\Gamma _{k}\left( \lambda n+\nu+1\right) \left( n!\right) ^{2}}\]
is log-convex.

A logarithmic differentiation of $a_n(\nu)$ with respect to $\nu$ yields
\begin{align*}
\frac{\partial}{\partial \nu} \log\left(a_{n}^{k}\left( \nu \right)\right)=\Psi_k\left(\nu+1\right) - \Psi_k\left( \lambda n+\nu+1\right).
\end{align*}
This along with \eqref{def-digamma-2} gives
\begin{align*}
\frac{\partial^2}{\partial\nu^2}\log\left(a_{n}^{k}\left( \nu \right)\right)
&=\Psi'_k\left(\nu+1\right) - \Psi'_k\left( \lambda n+\nu+1\right)\\
&=\sum_{r=0}^\infty  \frac{1}{(rk+\nu+1)^2} - \sum_{r=0}^\infty  \frac{1}{(rk+\lambda n+\nu+1)^2}\\
&=\sum_{r=0}^\infty  \frac{\lambda n(2 rk+\lambda n+2\nu+2)}{(rk+\nu+1)^2(rk+\lambda n+\nu+1)^2} >0,
\end{align*}
for all $n \geq 0$, $k >0$ and $\nu>-1$. Thus, $\nu \mapsto a_{n}^{k}\left( \nu \right)$ is log-convex and hence the conclusion. \\

\noindent
{\bf (4).}
Denote  $\Phi _{k}\left( a,c;x\right)=\sum_{n=0}^\infty c_{n, k}( a, c) x^{n}$
and  $\mathtt{I}_{ k, \nu}^{\gamma ,\lambda }(x)= \sum_{n=0}^\infty a_n(\nu) x^n,$
where
\begin{equation*}
a_n(\nu)= \frac{(\gamma)_{n,k}}{\Gamma_k(\lambda n+\nu+1) (n!)^2 2^n}\quad \text{and} \quad 
d_{n,k}\left( a,c\right) =\frac{\left( a\right) _{n,k}}{\left( c\right)
_{n,k}n!}
\end{equation*}%
with $v>-1$ and $a,c, \lambda, \gamma, k>0.$ To apply Lemma \ref{lemma:1}, consider the sequence
$\left\{ w_{n}\right\} _{n\geq 0}$ defined by
\begin{eqnarray*}
w_{n} =\frac{a_{n}\left( \nu \right) }{d_{n, k}\left( a,c\right) }&=&\frac{%
\Gamma _{k}\left( \gamma +nk\right) }{2^{n}\Gamma _{k}\left( \gamma \right)
\Gamma _{k}\left( \lambda n+\alpha +1\right) \left( n!\right) ^{2}}.\frac{%
\Gamma _{k}\left( a\right) \Gamma _{k}\left( c+nk\right) n!}{\Gamma
_{k}\left( a+nk\right) \Gamma _{k}\left( c\right) } \\
&=&\frac{\Gamma _{k}\left( a\right) }{\Gamma _{k}\left( \gamma \right)
\Gamma _{k}\left( c\right) }\rho _{k}\left( n\right)
\end{eqnarray*}%
where
\begin{equation*}
\rho _{k}\left( x\right) =\frac{\Gamma _{k}\left( \gamma +xk\right) \Gamma
_{k}\left( c+xk\right) }{\Gamma _{k}\left( \lambda x+\nu +1\right) \Gamma
_{k}\left( a+xk\right) 2^x \Gamma(x+1) }.
\end{equation*}%
In view of the increasing properties of $\Psi_k$ on $(0, \infty)$, and
\begin{equation*}
\frac{\rho ^{\prime }\left( x\right) }{\rho \left( x\right) }= k \psi
_{k}\left( \gamma +xk\right) +k\psi _{k}\left( c+xk\right) -\lambda \psi _{k}\left(
\lambda x+\alpha +1\right) -k \psi _{k}\left( a+xk\right),
\end{equation*}
it follows that for  $a\geq c>0$, $\lambda \geq k$ and $\nu+1\geq \gamma$,  the function $\rho$ is decreasing on $
\left( 0,\infty \right) $ and thus the sequence $\left\{ w_{n}\right\} _{n\geq
0} $ also decreasing. Finally the conclusion for $(a)$ follows from the Lemma \ref{lemma:1}.

In the case $(b)$ and $(c)$, the sequence $\{w_n\}$ reduces to 
\begin{eqnarray*}
w_{n} =\frac{a_{n}\left( \nu \right) }{d_{n, k}\left(\gamma, \lambda\right) }&=&\frac{\rho _{k}\left( n\right)}{\Gamma _{k}\left( \lambda \right)
 }
\end{eqnarray*}%
where
\begin{equation*}
\rho _{k}\left( x\right) =\frac{\Gamma _{k}\left( \lambda +xk\right) }{\Gamma _{k}(\nu +1+\lambda x) \Gamma(x+1) }.
\end{equation*}%
Now as in the proof of part (a)  
\begin{equation*}
\frac{\rho _{k}'\left( x\right)}{\rho _{k}\left( x\right)} =k \Psi_k(\lambda +xk)-\lambda \Psi_k(\nu+1 +xk)-\Psi(x+1)>0,
\end{equation*}
if $\nu+1+ \lambda x \geq \lambda + xk$. Now for $x \in (0,1)$, this inequality holds  if $0< k \leq \lambda \leq \nu+1$, while for $x \geq 1$,  it is required that $k \leq \min\{\lambda, \nu+1\}.$
\end{proof}

\end{document}